\documentclass{article}
\usepackage[utf8]{inputenc}
\usepackage[margin=1in]{geometry}
\usepackage{amsmath,amsthm,amssymb}
\usepackage{enumerate}
\usepackage{xfrac}
\usepackage[colorlinks=true,hidelinks]{hyperref}
    \AtBeginDocument{}
\usepackage{tikz}
\usetikzlibrary{matrix,arrows,decorations.pathmorphing,backgrounds,automata}

\usepackage{bbold} 

\usepackage{fancyhdr}
\newcommand{\J}{\mathcal{J}}

\newcommand{\Z}{\mathbb{Z}}

\newcommand{\inv}{^{-1}}

\newcommand{\im}{\ensuremath{\operatorname{im}}}

\newcommand{\Aut}{\ensuremath{\operatorname{Aut}}}

\newcommand{\Gal}{\ensuremath{\operatorname{Gal}}}

\newcommand{\Div}{\operatorname{Div}}
\newcommand{\Prin}{\operatorname{Pr}}

\newtheorem{thm}{Theorem}[section]
\newtheorem{prop}[thm]{Proposition}

\newtheorem{defn}[thm]{Definition}

\newtheorem{rem}[thm]{Remark}
  \let\oldrem\rem
  \renewcommand{\rem}{\oldrem\normalfont}

\newtheorem{ex}[thm]{Example}
  \let\oldex\ex
  \renewcommand{\ex}{\oldex\normalfont}

\makeatletter
\let\@fnsymbol\@arabic
\makeatother

\title{Jacobians of Graphs via Edges and Iwasawa Theory}
\author{Jon Aycock} 
\date{\today}

\begin{document}

\maketitle

\rhead{\thepage}
\cfoot{}

\begin{abstract}
The Jacobian is an algebraic invariant of a graph which is often seen in analogy to the class group of a number field. In particular, there have been multiple investigations into the Iwasawa theory of graphs with the Jacobian playing the role of the class group. In this paper, we construct an Iwasawa module related to the Jacobian of a $\Z_p$-tower of connected graphs, and give examples where we use this to compute asymptotic sizes of the Jacobians in this tower.
\end{abstract} 

\section{Introduction}
\label{sec:intro}

The Jacobian is an algebraic invariant of a graph which is often seen in analogy to the class group of a number field. Classically, the Jacobian is defined in terms of the vertices of the graph; however, in this paper we give another construction of the Jacobian in terms of the edges in Definition \ref{def:edgejacobian}. We prove that this gives an isomorphic group in Proposition \ref{thm:comparejacobians}.

Because of the analogy to the class group, there have been investigations into a version of Iwasawa theory for graphs. In \cite{MV1,MV2}, the authors $p$-adically interpolate the Ihara zeta function and use an analog of the analytic class number formula to get asymptotics for the sizes of Jacobians in a $\Z_p$-tower of graphs. A version of the Iwasawa Main Conjecture is explored in \cite{KM}. Special values of Artin--Ihara zeta functions are also used in \cite{HMSV} to prove analogs of some other results in class field theory. In \cite{Gonet,Kataoka}, the authors instead construct a $\Lambda$-module which is related to the Jacobian. This is more in line with the methods of this paper; the main result here is Theorem \ref{thm:themodule}, which is roughly as follows.

\begin{thm}
Let $\Gamma$ be a group isomorphic to $\Z_p$, $\Gamma_n$ be its quotient isomorphic to $\Z/p^n\Z$, and let $\Gamma^{(n)}$ be the kernel of the projection $\Gamma \to \Gamma_n$. There is a module $J(\infty)$ over the Iwasawa algebra $\Lambda = \Z_p[[\Gamma]]$ whose coinvariants $J(n) = J(\infty) \otimes_\Lambda \Z_p[\Gamma_n]$ are isomorphic to an extension of the $p$-power torsion subgroup of the Jacobian $\J(X_n)[p^\infty]$ by $\Gamma^{(n)} \cong \Z_p$.
\end{thm}

One may contrast this with \cite[Proposition 4]{Gonet}, which states that the $\Lambda$-module constructed in that paper is an extension of $\Z_p$ by $\J(X_n)[p^\infty]$. In addition to the extra copy of $\Z_p$ being a submodule instead of a quotient, the way in which these extra factors arise are different. In \cite{Gonet}, one must control for the degree of the divisor. However, in Theorem \ref{thm:themodule}, the extra factor of $\Z_p$ comes from Galois theory; we first mention this defect in Section \ref{sec:homologyincovers}, and we talk about it in more detail in Remark \ref{rem:actiononliftedspanningtrees}.

In Section \ref{sec:jacobians}, we explore the classical theory of the Jacobian of a graph. In Section \ref{sec:towers}, we give some facts about the Galois theory of covering graphs, which we use to construct the module in Section \ref{sec:module}. Finally, we use this module (with some small modifications that control for the fact that $J(n)$ is a proper extension of the group we want to count) to actually get the asymptotics for the size of $\J(X_n)[p^\infty]$ for a specific example in Section \ref{sec:examples}.

\subsection{Acknowledgements}
The author would like to thank Daniel Vallieres for his introduction to the topic. 
In addition, the author would like to thank the person who produced the figures, though that person would like to remain anonymous.

\section{Jacobians}
\label{sec:jacobians}

\subsection{Graphs}
Im this paper, all of our graphs will be finite, connected, and undirected. Following conventions from e.g. \cite{Gonet,Terras}, we define a graph as follows:
\begin{defn}
A graph is a tuple $X = (V,\vec{E},o,t,\iota)$ of a set $V$ of vertices and a set $\vec{E}$ of directed edges with incidence maps $o,t \colon \vec{E} \to V$. The graph is finite if both $V$ and $\vec{E}$ are finite sets. A graph undirected if it is equipped with a fixed-point-free involution $\iota \colon \vec{E} \to \vec{E}$ with the properties that $o(\iota(e)) = t(e)$ and $t(\iota(e)) = o(e)$.
\end{defn}
We think of the edge $e$ as connecting the vertex $o(e)$ to the vertex $t(e)$, so that $\iota(e)$ runs the other way, connecting $t(e)$ to $o(e)$. The pair of directed edges $\{e,\iota(e)\}$ is thought of as a single undirected edge connecting the two vertices. We will also fix an orientation of $X$, meaning a partition $\vec{E} = E^+ \sqcup E^-$ with $\iota(E^+) = E^-$. We refer to $E^+$ as being the set of positively oriented edges.

A path on $X$ is a sequence of edges $e_1\cdots e_\ell$ with $o(e_{j+1}) = t(e_j)$ for each $i$. If, for any pair of vertices $v_1$ and $v_2$, there is a path $e_1\cdots e_\ell$ with $o(e_1) = v_1$ and $t(e_\ell) = v_2$, we say that $X$ is connected.

\begin{rem}
What we call a graph is often called a multigraph, e.g. in \cite{HMSV}. We will instead give the name ``simple graph" to what they call a graph. For ease of translation between papers, we give the definition of a simple graph here.

A loop on $X$ is an edge $e$ with $o(e) = t(e)$. Two distinct edges $e \neq e^\prime$ are said to be parallel if $o(e) = o(e^\prime)$ and $t(e) = t(e^\prime)$. A graph is simple if it has no loops and no two of its edges are parallel.
\end{rem}

Finally, we define what we mean by a morphism of graphs.
\begin{defn}
Let $X = (V_X,\vec{E}_X,o_X,t_X,\iota_X)$ and $Y = (V_Y,\vec{E}_Y,o_Y,t_Y,\iota_Y)$ be two graphs. A morphism $\phi \colon X \to Y$ is a pair of functions $\phi = (\phi_V,\phi_E)$ with $\phi_V \colon V_X \to V_Y$ and $\phi_E \colon \vec{E}_X \to \vec{E}_Y$ satisfying two compatibility relations. First, $\phi_E(\iota_X(e)) = \iota_Y(\phi_E(e))$. In addition, we should also have $\phi_V(o_X(e)) = o_Y(\phi_E(e))$ and $\phi_V(t_X(e)) = t_Y(\phi_E(e))$.
\end{defn}
Note that the morphism $\phi$ is determined by the choice of $\phi_E$, and in fact by its restriction to $E_X^+$. We say that a morphism is oriented if every positively oriented edge in $X$ is sent to a positively oriented edge in $Y$; i.e., if $\phi_E({E}_X^+) \subset {E}_Y^+$.

\subsection{Divisors and Principal Divisors}
The divisor group of $X$ is group of formal linear combinations of vertices, $\Div(X) = \Z V$. We have a degree homomorphism $\operatorname{deg} \colon \Div(X) \to \Z$ with the formula
\begin{equation}
	\operatorname{deg}\left( \sum_{v \in V} a_vv \right) = \sum_{v \in V} a_v.
\end{equation}
The kernel of $\operatorname{deg}$ is the group of degree zero divisors, $\operatorname{Div}^0(X)$.

For any pair of vertices $v \neq v^\prime \in V$, let $\operatorname{deg}(v,v^\prime)$ be the number of edges connecting $v$ to $v^\prime$.\footnote{Note that, if $X$ is a simple graph, $\operatorname{deg}(v,v^\prime) \in \{0,1\}$ for any pair of vertices $v$ and $v^\prime$. We make no restriction to simple graphs, so it is possible for multiple edges to connect the same pair of vertices. How we handle the existence (or non-existence) of loops will be inconsequential, since we will never ask for $\operatorname{deg}(v,v)$.} Then let $\operatorname{deg}(v) = \sum_{v^\prime \neq v} \operatorname{deg}(v,v^\prime)$ be the number of edges incident to $v$.
The Laplacian map $\Delta \colon \Div(X) \to \Div(X)$ is determined by the values, and extended $\Z$-linearly:
\begin{equation}
	\Delta(v) = \deg(v)v - \sum_{v^\prime \neq v} \operatorname{deg}(v,v^\prime)v^\prime.
\end{equation}
The image of $\Delta$ is known the group of principal divisors, and denoted $\Prin(X)$. Note that $\Prin(X)$ is contained in $\Div^0(X)$, since $\operatorname{deg}(\Delta v)$ is $0$ for all $v$. Then we can define:

\begin{defn}\label{defn:vertexjacobian}
The vertex Jacobian of the graph $X$ is the quotient $\J_v(X) = \Div^0(X)/\Prin(X)$.
\end{defn}

\begin{rem}
This definition should be seen in analogy to the definition of the Picard group for a curve. 
For intuition on why one might think that the image of the Laplacian should give the principal divisors, see e.g. \cite[Remark 1.4]{Baker}.
\end{rem}

We can compute the size and the invariant factors of the Jacobian of a graph by looking at the matrix for its Laplacian.

\begin{ex}\label{ex:vertex-jacobian-of-graph-1}
Consider the graph $X_{a,b}$ with three vertices $V = \{v_1,v_2,v_3\}$. We connect $v_1$ to $v_2$ with $a$ undirected edges labeled $e_1,\dots,e_a$, and connect $v_2$ to $v_3$ with $b$ undirected edges labeled $e_1^\prime,\dots,e_b^\prime$. In the ordered basis $v_1,v_2,v_3$, the Laplacian $\Delta_{X_{a,b}}$ has matrix
\begin{equation}
	\Delta = \begin{pmatrix}
		a & -a & 0 \\
		-a & a+b & -b \\
		0 & -b & b
	\end{pmatrix}
\end{equation}
This matrix is quickly seen to be singular. However, this is simply a manifestation of the fact that its image is contained in the subgroup $\Div^0(X)$ of $\Div(X)$, which has infinite index. The standard trick to modify this matrix in order to obtain a nonsingular matrix whose determinant is the size of $\J_v(X)$ is to pick a vertex to be a ``sink;" we will pick $v_2$ to be the sink, and this gives rise to an isomorphism $\Z v_1 \oplus \Z v_3 \xrightarrow{\sim} \Div^0(X)$ by the map $a_1v_1 + a_3v_3 \mapsto a_1v_1 + (-a_1-a_3)v_2 + a_3v_3$. This allows us to view $\Delta$ as a map to $\Z v_1 \oplus \Z v_3$; its matrix will be the same as the original matrix for $\Delta$, but with the row corresponding to $v_2$ omitted. If we consider the minor that we get by also omitting the column corresponding to $v_2$, we find that the cokernel of $\Delta$ will have order $\det \begin{pmatrix}
	a & 0 \\ 0 & b
\end{pmatrix} = ab$. Since $\J_v(X_{ab})$ is the cokernel of $\Delta$ viewed as a map to $\Div^0(X_{a,b})$ $\#\J_v(X_{a,b}) = ab$.
\end{ex}

\begin{rem}
The minor in Example \ref{ex:vertex-jacobian-of-graph-1} is independent of which column we omit, up to sign. Omitting the column corresponding to $v_2$ just gives the nicest remaining matrix, since it is diagonal. In fact, the minors of $\Delta_{X_{a,b}}$ obtained from omitting any single row and any single column are determined (up to sign) independently of which row and which column you omit. This is some of the content of Kirchhoff's Matrix-Tree Theorem.
\end{rem}

\begin{ex}\label{ex:vertex-jacobian-of-graph-2}
In this example, we compute the size of the Jacobian of the complete graph on $4$ vertices, $K_4$. Its Laplacian is the following matrix:
\begin{equation}
	\Delta = \begin{pmatrix}
		3 & -1 & -1 & -1 \\
		-1 & 3 & -1 & -1 \\
		-1 & -1 & 3 & -1 \\
		-1 & -1 & -1 & 3
	\end{pmatrix}
\end{equation}
As in the previous example, we can omit any single row and any single column and the determinant of the resulting matrix will give us the order of the cokernel of $\Delta$ as a map to $\Div^0(X)$, which is then the order of $\J_v(K_4)$. A calculation of invariant factors in Sage gives that $\J_v(K_4) \cong (\Z/4\Z)^2$, so that $\#\J_v(K_4) = 16$.
\end{ex}

\subsection{Stars and Homology}
In Section \ref{subsec:edges}, we will give another construction of the Jacobian of $X$. Here, we define the pieces that go into that discussion.
There are two important maps between $\Z V$ and $\Z E^+$ that we want to use.

\begin{defn}
The boundary map $d \colon \Z E^+ \to \Z V$ is determined by the values $d(e) = t(e) - o(e)$ and extended linearly. The kernel $\ker d$ is the first homology group of $X$, $H_1(X,\Z)$. For any prime number $p$, we can define $H_1(X,\Z_p) = H_1(X,\Z) \otimes_\Z \Z_p$.
\end{defn}

\begin{defn}
The map $s \colon \Z V \to \Z E^+$ is determined by the following values and extended linearly. The sums are over edges in $E^+$.
\begin{equation}
	s(v) = \sum_{t(e) = v} e - \sum_{o(e) = v} e.
\end{equation}
The image $\im s$ is denoted by $S(X,\Z)$. For any prime number $p$, we can define $S(X,\Z_p) = S(X,\Z) \otimes_\Z \Z_p$.
\end{defn}

The group $H_1(X,\Z)$ is actually the first homology group of $X$ viewed as a topological space. However, $S(X,\Z)$ is a more combinatorial object that is specific to the situation of a graph. We will refer to it as the group of stars, since it is generated by the elements $s(v)$, and $s(v)$ is the star into the vertex $v$.

Note that the composition $d \circ s \colon \Z V \to \Z V$ is the Laplacian $\Delta$.

\subsection{The Jacobian via Edges}\label{subsec:edges}
\begin{defn}\label{def:edgejacobian}
The edge Jacobian of the graph $X$ is the quotient $\J_e(X) = \Z E^+/[H_1(X,\Z)+S(X,\Z)]$.
\end{defn}

\begin{prop}\label{thm:comparejacobians}
The vertex Jacobian $\J_v(X)$ and the edge Jacobian $\J_e(X)$ are isomorphic.
\end{prop}
\begin{proof}
The isomorphism will be induced by the boundary map $d$. Note that, since $X$ is connected, the image of $d$ is $\Div^0(X)$, so $d$ induces an isomorphism $\Z E^+/H_1(X,\Z) \to \Div^0(X)$. In addition, the Laplacian $\Delta \colon \Z V \to \Z V$ factors as the composition $d \circ s \colon \Z V \to \Z E^+ \to \Z V$; the image $\Prin(X)$ of $\Delta$ is precisely the image of $S(X,\Z)$ under $d$. Thus we have that $\J_e(X) = \Z E^+/[H_1(X,Z) + S(X,\Z)] \cong \Div^0(X)/\Prin(X) = \J_v(X)$, as desired.
\end{proof}

In light of this proposition, we will write $\J(X)$ for the edge Jacobian. This will be better to set up the constructions in Section \ref{sec:module}, and it will not contradict any previous papers that use the vertex Jacobian.

The two methods to defining the Jacobian each have their advantages. The vertex Jacobian gives an easier ability to compute the size of the group as in Examples \ref{ex:vertex-jacobian-of-graph-1} and \ref{ex:vertex-jacobian-of-graph-2}. However, the edge Jacobian's prioritization of the edges over the vertices aligns better with the theory of the Ihara zeta function; in this theory the primes on $X$ (i.e., the points of the space corresponding to $X$) are paths on $X$ rather than vertices. One can see this in e.g. \cite{Terras}.

\subsection{Spanning Trees}\label{sec:spanningtrees}
The Jacobian $\J(X)$ has a combinatorial meaning in counting specific subgraphs of $X$.

\begin{defn}
A graph $T$ is called a tree if it is connected and $H_1(T,\Z) = 0$. A spanning tree of $X$ is a subgraph of $X$ which is a tree, and which includes every vertex of $X$.
\end{defn}

\begin{ex}
Consider the graphs $X_{a,b}$ from Example \ref{ex:vertex-jacobian-of-graph-1}. Any spanning tree in $X_{a,b}$ needs to be spanning, meaning that it needs to contain all three vertices from $X_{a,b}$; and it needs to be a tree, meaning that it needs to be connected and contain no cycles. In order to specify such a subgraph, we just need to specify the edges. This involves making two independent choices: first, choose one edge out of the $a$ that connect $v_1$ to $v_2$, and then choose one edge out of the $b$ that connect $v_2$ to $v_3$. This gives a total of $ab$ spanning trees.
\end{ex}

\begin{ex}\label{ex:spanningK4}
As in Example \ref{ex:vertex-jacobian-of-graph-2}, we now consider the complete graph on $4$ vertices, $K_4$. All $16$ spanning trees are shown in Figure \ref{fig:spanningk4}.

\begin{figure}
\centering
\includegraphics[width=4in]{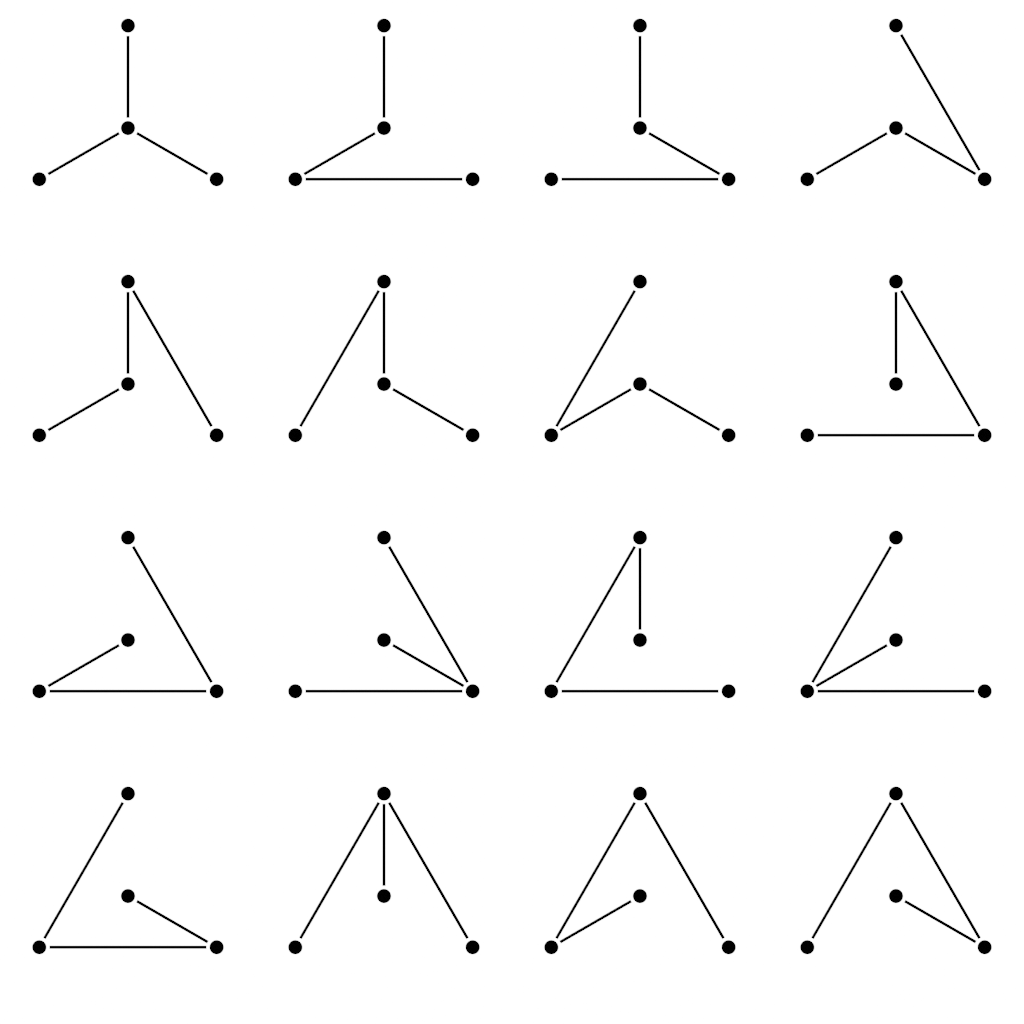}
\caption{The $16$ spanning trees on $K_4$, as in Example \ref{ex:spanningK4}.}
\label{fig:spanningk4}
\end{figure}
\end{ex}

Recall that $\#\J_v(X_{a,b}) = ab$ and $\#\J_v(K_4) = 16$. This suggests the following theorem:

\begin{thm}
Let $\kappa(X)$ be the number of spanning trees of $X$. Then $\kappa(X) = \#\J(X)$.
\end{thm}

In fact, we can prove this by realizing the set of spanning trees of $X$ as a torsor for $\J(X)$. See \cite{CCG} for a description of this process, which is known as rotor-routing.

\begin{rem}\label{rem:SpanningTreesAreATorsor}
It is the point of view of this paper that viewing the Jacobian in terms of edges is often profitable. However, this is one situation in which it is highly unclear how this is helpful. The process of rotor-routing essentially uses the vertices, and even relies on a choice of a base vertex. In \cite{CCG}, the action is described only on divisors of the form $(x-r)$ where $r$ is the chosen base vertex and $x$ is another vertex, and extended linearly. It is not described in a way that makes it clear what the action of a single edge $e$ should be, though we could describe it in terms of paths that end at the root. Even more, changing the base vertex changes the action (for non-planar graphs), which makes it even less likely that rotor-routing, as is, could be modified directly in such a way that it gives an action of $\J_e(X)$ without appealing to the isomorphism with $\J_v(X)$.

It seems difficult in general to give an action of divisors on spaces of subgraphs. We ask here whether or not there is a different process that gives an action of $\J_e(X)$ intrinsically; later in Remark \ref{rem:actiononliftedspanningtrees} we will ask a similar question for a similar group.
\end{rem}

\section{Covers of Graphs}
\label{sec:towers}

In this section, we discuss the Galois theory of graphs. We follow \cite{Gonet} for much of the exposition, though more detail is present in \cite{Terras}.

\subsection{Covering Graphs}
\begin{defn}
For an oriented graph $X$ and a chosen vertex $v \in V$, write $S_v^+(X) = \{ e \in E^+ \mid o(e) = v \text{ or } t(e) = v \}$.

Let $\tilde{X}$ and $X$ be connected graphs. We say that a morphism $\pi \colon \tilde{X} \to X$ is a covering map if the following two conditions are true.
\begin{itemize}
	\item $\pi_V \colon V_{\tilde{X}} \to V_X$ is surjective, and
	\item for each vertex ${v} \in V_{\tilde{X}}$, the restriction of $\pi_E$ to $S_{{v}}^+(X)$ gives a bijection $S_{{v}}^+(X) \xrightarrow{\sim} S_{\pi_V({v})}^+(B)$.
\end{itemize}
\end{defn}
\begin{rem}
Note that the second condition implicitly requires that $\pi$ be oriented. If $X$ is oriented but $\tilde{X}$ is unoriented, one can weaken that condition, and then induce an orientation on $\tilde{X}$ from the orientation on $X$ that makes $\pi$ a covering map as follows.

Write $S_v^\pm(\tilde{X}) = \{ e \in \vec{E}_{\tilde{X}} \mid o(e) = v \text{ or } t(e) = v \}$, and similarly for $X$. Then the second condition is that $\pi$ gives a bijection $S_v^\pm(\tilde{X}) \to S_{\pi_V(v)}^\pm(X)$. Then we pick $E_{\tilde{X}}^+$ to be the set of edges $e$ such that $\phi_E(e) \in E_X^+$.
\end{rem}

The theory of covering maps is a version of Galois theory. We refer the reader to \cite{Terras} for a more in depth version of the discussion that follows, including proofs of the main theorems. We record the following few facts and notations for later use.

Let $\pi \colon \tilde{X} \to X$ be a covering map. The automorphism group of $\pi$ is the group $\Aut(\tilde{X}/X) = \{\sigma \in \Aut(\tilde{X}) \mid \pi \circ \sigma = \pi \}$. If $\tilde{X}$ and $X$ are both finite graphs, then $\pi$ is $n$-to-one for some $n$. We call this the degree, and denote it $[\tilde{X}:X]$. When $\tilde{X}$ is connected, we have the bound $\#\Aut(\tilde{X}/X) \leq [\tilde{X}:X]$; if $\#\Aut(\tilde{X}/X) = [\tilde{X}:X]$, then we say $\pi$ is Galois.

\begin{defn}
Let $\pi \colon Y \to X$ and $\rho \colon Z \to X$ be two covering maps. A morphism of covering maps is a graph morphism $\phi \colon Y \to Z$ such that $\pi = \rho \circ \phi$. If $\phi$ is invertible, we say that $\pi$ and $\rho$ are isomorphic. If $\phi$ is surjective, we say that $\rho$ is a subcover of $\pi$.
\end{defn}

Note that a morphism of covering graphs is necessarily oriented. Compare this definition with Definition 7 from \cite{Gonet}.

\subsection{Constructing Covers: Voltage Graphs}
We construct covering graphs by using the following tool.

\begin{defn}
A $G$-valued voltage assignment for a graph $X$ is a function $\alpha \colon \vec{E}_X \to G$ satisfying $\alpha(\iota(e)) = \alpha(e)\inv$.
\end{defn}

If the graph $X$ is oriented, the voltage assignment function $\alpha$ is determined by its values on $E^+$. Given a $G$-valued voltage assignment $\alpha$ for a graph $X$, we may refer to the triple $(X,G,\alpha)$ as a voltage graph.

The purpose of a voltage assignment is to build the derived graph; if it is connected, the derived graph will be a Galois cover of the base graph $X$ with Galois group $G$. Let $X = (V,\vec{E},o,t,\iota)$ be a graph, and fix a $G$-valued voltage assignment $\alpha$ for $X$. The derived graph is the graph $\tilde{X} = (\tilde{V},\tilde{E},\tilde{o},\tilde{t},\tilde{\iota})$ where $\tilde{V} = V \times G$, $\tilde{E} = \vec{E} \times G$, and the incidence functions are defined by the following formulas:
\begin{equation}
	\tilde{o}(e,g) = (o(e),g), \qquad \tilde{t}(e,g) = (t(e),g\alpha(e)), \qquad \tilde{\iota}(e,g) = (\iota(e),g\alpha(e)).
\end{equation}

\begin{rem}\label{rem:derivedgraphisgalois}
We get a graph morphism $\pi \colon \tilde{X} \to X$ with $\pi_V(v,g) = v$ and $\pi_E(e,g) = e$. If $G$ is finite, this has degree $[\tilde{X}:X] = \# G$. There is an action of $G$ on $\tilde{X}$ given by the formulas $h \cdot (v,g) = (v,hg)$ and $h \cdot (e,g) = (e,hg)$; this embeds $G$ into $\Aut(\tilde{X}/X)$. If $\tilde{X}$ is connected, then this shows that $\tilde{X}$ is a Galois cover of $X$.
\end{rem}

\begin{rem}
For the derived graph $\tilde{X}$ of a voltage graph $(X,G,\alpha)$ to be connected, it is neccessary but not sufficient that the image of $\alpha$ generate $G$. For example, if we drop the condition that $G$ be finite, there is a universal abelian voltage assignment $\alpha \colon \vec{E} \to \Z E^+$. Certainly the image of $\alpha$ generates $\Z E^+$, but in this case, there will be a path from the vertex $(v,0)$ to the vertex $(v,g)$ only if $g \in H_1(X,\Z) \subset \Z E^+$.
\end{rem}

We conclude this section with the statement of \cite[Theorem 4]{Gonet}, which explains the focus on voltage assignments for the construction of Galois covers.

\begin{thm}\label{thm:gonet4}
Let $(X,G,\alpha)$ be a voltage graph with $\tilde{X}$ the derived graph. If $\tilde{X}$ is connected, then $\tilde{X}/X$ is a normal cover with $\Gal(\tilde{X}/X) \cong G$. Conversely, given a Galois cover $\pi \colon \tilde{X} \to X$, then $\tilde{X}$ is the derived graph for some $G$-valued voltage assignment $\alpha$ on $X$.
\end{thm}

\subsection{Intermediate Voltage Covers}\label{sec:generaltowers}
Given a group homomorphism $f \colon G \to H$ and a voltage graph $(X,G,\alpha)$, we can build a new voltage graph $(X,H,f \circ \alpha)$. There is a morphism on the level of derived graphs given by applying $f$ to the second component. Write $\tilde{X}_G$ for the derived graph of the voltage assignment $\alpha$ and $\tilde{X}_H$ for the derived graph of the voltage assignment $f \circ \alpha$; if $\tilde{X}_G$ is connected and $f$ is surjective, then $\tilde{X}_H$ will be connected as well, and it will be an intermediate cover via the morphism of covering graphs described above.

\begin{ex}\label{ex:intermediatevoltagecovers}
Let $X = B_2 = (V,\vec{E},o,t,\iota)$ be the bouquet of two loops; here $V = \{*\}$ and $\vec{E} = \{e_1,\overline{e}_1,e_2,\overline{e}_2\}$, so that there are two undirected edges, each of which is a loop on the unique vertex. Let $G = \Z/4\Z$, and fix the voltage assignment $\alpha(e_1) = \alpha(e_2) = 1$. The derived graph $\tilde{X}$ has vertices $\tilde{V}=\{v_0,v_1,v_2,v_3\}$ indexed by $G$; the edges are as in the table below, with only the forward oriented edges listed.
\begin{equation}
	\begin{array}{c|c c c c c c c c}
		e & e_{1,0} & e_{1,1} & e_{1,2} & e_{1,3} & e_{2,0} & e_{2,1} & e_{2,2} & e_{2,3} \\
		\hline
		o(e) & v_0 & v_1 & v_2 & v_3 & v_0 & v_1 & v_2 & v_3 \\
		t(e) & v_1 & v_2 & v_3 & v_0 & v_1 & v_2 & v_3 & v_0
	\end{array}
\end{equation}
There is one intermediate cover $X_I$ of degree $2$ over $X$, given by the projection from $G$ to $\Z/2\Z$. We get two vertices $V_I = \{w_0,w_1\}$ indexed by the elements of $\Z/2\Z$, and four undirected edges; they are attached as in the table below, with only the forward oriented edges listed.
\begin{equation}
	\begin{array}{c|c c c c}
		e & e_{1,0} & e_{1,1} & e_{2,0} & e_{2,1}  \\
		\hline
		o(e) & w_0 & w_1 & w_0 & w_1 \\
		t(e) & w_1 & w_0 & w_1 & w_0 
	\end{array}
\end{equation}
The covering morphism sends both vertices to the unique vertex $*$ of $X$, and sends the edges $e_{i,g}$ to the loop $e_i$. See Figure \ref{fig:intermediatecovers} for a visual representation of the three graphs $\tilde{X}$, $X_I$, and $X$.

\begin{figure}
\centering
\includegraphics[width=5in]{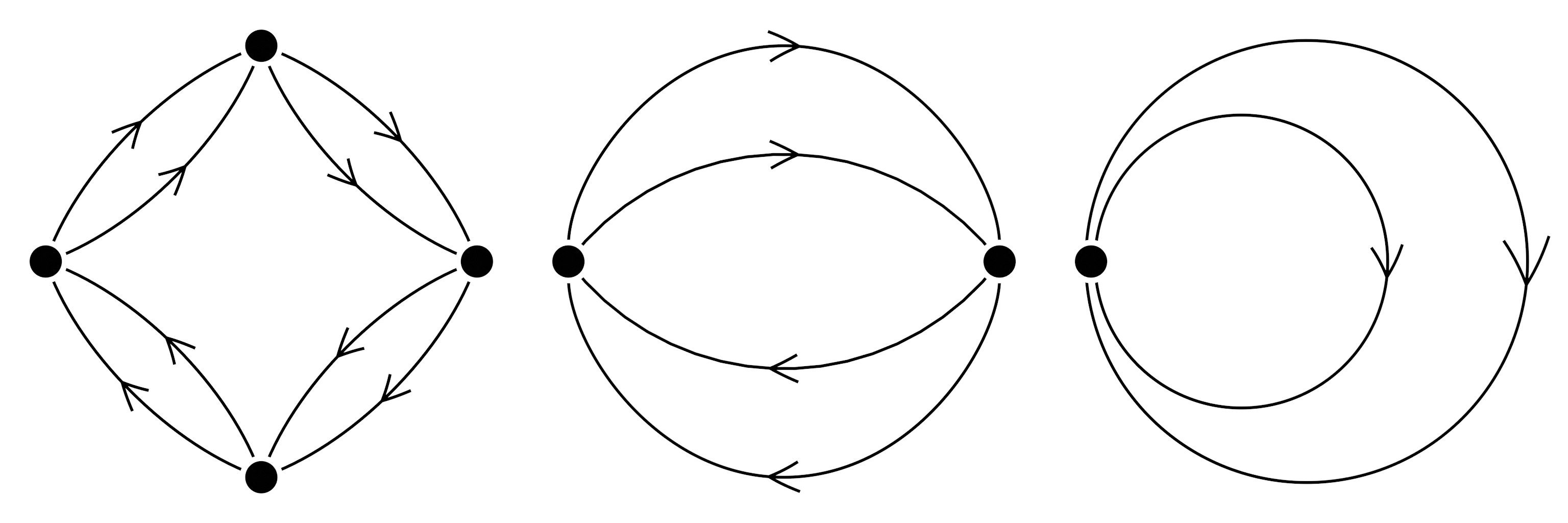}
\caption{From left to right, the graphs $\tilde{X}$, $X_I$, and $X$ from Example \ref{ex:intermediatevoltagecovers}.}
\label{fig:intermediatecovers}
\end{figure}
\end{ex}

\subsection{Relationships Between Homology Groups for Abelian Covers}\label{sec:homologyincovers}
The absolute Galois group of a graph $X$ is isomorphic to the profinite completion of its fundamental group $\pi_1(X,v_0)$. If we restrict to abelian covers $\pi \colon \tilde{X} \to X$, we will always have that $\Gal(\tilde{X}/X)$ is a quotient of the abelianization $\pi_1(X,v_0)^{ab} \cong H_1(X,\Z)$. In fact, we can describe exactly which quotient: pushforward along $\pi$ gives a morphism $\pi_* \colon H_1(\tilde{X},\Z) \to H_1(X,\Z)$, whose cokernel is naturally isomorphic to $\Gal(\tilde{X}/X)$.

\begin{ex}
Let $X = B_2$, and $\tilde{X}$ be the Galois cover of $X$ with Galois group $\Z/4\Z$ as in Example \ref{ex:intermediatevoltagecovers} above.

We have that $H_1(\tilde{X},\Z) \cong \Z^5$, and it is generated by the following cycles:
\begin{equation}
	c_0 = e_{1,0} - e_{2,0}, \quad c_1 = e_{1,1} - e_{2,1}, \quad c_2 = e_{1,2} - e_{2,2}, \quad c_3 = e_{1,3} - e_{2,3}, \quad 
\end{equation}
along with the cycle $c^\prime = e_{1,0}+e_{1,1}+e_{1,2}+e_{1,3}$. Pushed forward, these become the following elements of $H_1(X,\Z) = \Z e_1 \oplus \Z e_2$:
\begin{equation}
	\begin{array}{c|ccccc}
		c & c_0 & c_1 & c_2 & c_3 & c^\prime \\
		\hline
		\pi_*(c) & e_1 - e_2 & e_1 - e_2 & e_1 - e_2 & e_1 - e_2 & 4e_1
	\end{array}
\end{equation}
Their images generate the subgroup $\pi_*(H_1(\tilde{X},\Z)) = \langle e_1-e_2, 4e_1 \rangle$; the quotient of $H_1(X,\Z) = \Z\{e_1,e_2\}$ by this subgroup identifies $e_1$ with $e_2$ and then kills $4e_1$, so
\begin{equation}
H_1(X,\Z)/\pi_*(H_1(\tilde{X},\Z)) \cong \Z/4\Z.
\end{equation}
This was $\operatorname{Gal}(\tilde{X}/X)$ by construction.
\end{ex}

\section{The Iwasawa Module}
\label{sec:module}

\subsection{$\Z_p$-covers}
Classically, Iwasawa theory concerns itself with infinite field extensions, specifically those with Galois group isomorphic to $\Z_p$. In the present work, all graphs are assumed to be finite, which means that we will never have a covering map with infinite Galois group. However, we will still obtain some infinitary results by considering infinite towers of finite graphs.

\begin{defn}\label{def:ZpTower}
A $\Z_p$-tower of graphs is a collection of graphs $\{X_i\}_{i=0}^\infty$ and Galois covering morphisms $X_0 \gets X_1 \gets X_2 \gets \dots$ with the property that $\operatorname{Gal}(X_j/X_i) \cong \Z/p^{j-i}\Z$ for each $j \geq i$.
\end{defn}

\begin{rem}\label{rem:getZpTower}
We will create $\Z_p$-towers of graphs by assigning voltages in the infinite group $\Z_p$ and applying the method in Section \ref{sec:generaltowers} to produce the infinite tower. Careful application of Theorem \ref{thm:gonet4} shows that all $\Z_p$-towers of graphs can be obtained in this way.
\end{rem}

Write $\Gamma_n = \operatorname{Gal}(X_n/X_0)$, and let $\Gamma = \varprojlim \Gamma_n$, so that $\Gamma \cong \Z_p$. Fix a topological generator $\gamma$ of $\Gamma$, so that the class of $\gamma$ in the quotient $\Gamma_n$ will generate $\Gamma_n$. In addition, let $\Gamma^{(n)} = \overline{\langle \gamma^{p^n} \rangle} \subset \Gamma$ be the kernel of the projection to $\Gamma_n$. We let $\Lambda = \varprojlim \Z_p[\Gamma_n]$ be the completed group ring, which we also denote as $\Lambda = \Z_p[[\Gamma]]$. $\Lambda$ is isomorphic to the standard power series ring $\Z_p[[T]]$ by the continuous homomorphism $\gamma \mapsto T+1$.

\begin{rem}
Since we focus on finite graphs here, there is no graph $X_\infty = \varprojlim X_n$. However, there is a theory of profinite graphs into which this graph would fit. For intuition, one may think of an infinite graph $X_\infty$ living at the top of the $\Z_p$ tower; this is done in, e.g., \cite{Kataoka,KM}, though this is not necessary for what we do here. With this, we have the added intuition that $\Gamma^{(n)} = \operatorname{Gal}(X_\infty/X_n)$.
\end{rem}

\subsection{Cycles, Stars, and Homology}\label{sec:CSHinTowers}
As foreshadowed in the Remark \ref{rem:getZpTower}, we fix a voltage graph $(X_0,\Gamma,\alpha)$. Let $\alpha_n$ be the composition of $\alpha$ with the projection $\Gamma \to \Gamma_n$, and let $X_n$ be the derived graph of the voltage assignment $(X_0,\Gamma_n,\alpha_n)$; we assume that we have chosen the voltage assignment such that each $X_n$ is connected, in which case $X_n$ is a Galois cover of $X_0$ with Galois group $\Gamma_n \cong \Z/p^n\Z$. Thus we have a $\Z_p$-tower of graphs in the sense of Definition \ref{def:ZpTower}.

For each $n$, let $V_n$ be the set of vertices of $X_n$ and $\vec{E}_n$ its set of edges. Fix an orientation $E_0^+$ for $X_0$, and note that this determines an orientation $E_n^+$ for each $X_n$. For each $n$, write $C(n) = \Z_p E_n^+$ and $D(n) = \Z_p V_n$.

Now fix $n$, and for each vertex and each edge of $X_0$ pick a lift to $X_n$. Note that, for a fixed vertex $v_0 \in V_0$, there are exactly $p^n$ choices for the lift $v_n$, and for a fixed choice of the lift $v_n$, the set of choices is $\{g v_n \mid g \in \Gamma_n\}$. Similarly, for a fixed (positively oriented) edge $e_0 \in E_0^+$, there are exactly $p^n$ choices for the lift $e_n$, and for a fixed choice of the lift $e_n$, the set of choices is $\{g e_n \mid g \in \Gamma_n\}$. From this, we see that a choice of lifts gives an isomorphism $C(n) \cong \Z_p[\Gamma_n] E_0^+$ and $D(n) \cong \Z_p[\Gamma_n] V_0$.

Now fix a compatible system of lifts $v_n \in V_n$ for each $n$ such that, for every $m \leq n$, $v_n$ maps to $v_m$ under the projection from $X_n$ to $X_m$. Then do the same with each (positively oriented) edge $e_0 \in E_0^+$. The above line of reasoning, along with the compatibility of the lifts, gives a compatible system of isomorphisms $C(n) \cong \Z_p[\Gamma_n] E_0^+$ and $D(n) \cong \Z_p[\Gamma_n] V_0$.

Recall the maps $d \colon \Z E^+ \to \Z V$ and $s \colon \Z V \to \Z E^+$ from Section \ref{subsec:edges}; let $d$ and $s$ also denote the maps on the extensions of scalars to $\Z_p$ as in this section. Note that these functions are Galois equivariant --- in fact, the identifications $\Z_p E_n^+ \cong \Z_p[\Gamma_n] E_0^+$ and $\Z_p V_n \cong \Z_p[\Gamma_n] V_0$ give commuting diagrams
\begin{equation}\nonumber
	\begin{tikzpicture}
	
	\matrix (m) [matrix of math nodes,row sep = 2em,column sep = 3em,minimum width = 3em]
	{
	\Z_p E_n^+ & \Z_p V_n && \Z_p V_n & \Z_p E_n^+  \\
	\Z_p[\Gamma_n] E_0^+ & \Z_p[\Gamma_n] V_0 && \Z_p[\Gamma_n] V_0 & \Z_p[\Gamma_n] E_0^+  \\
	};
	\path[-stealth]
	(m-1-1) edge node [above] {$d$} (m-1-2)
	(m-2-1) edge node [above] {$d$} (m-2-2)
	(m-1-1) edge (m-2-1)
	(m-1-2) edge (m-2-2)

	(m-1-4) edge node [above] {$s$} (m-1-5)
	(m-2-4) edge node [above] {$s$} (m-2-5)
	(m-1-4) edge (m-2-4)
	(m-1-5) edge (m-2-5)	
	;
	\end{tikzpicture}
\end{equation}
We deduce from this that $H_1(X_n,\Z_p)$ can be computed as the kernel of $d \colon \Z_p[\Gamma_n] E_0^+ \to \Z_p[\Gamma_n]V_0$, and that $S(X_n,\Z_p)$ can be computed as the image of $s \colon \Z_p[\Gamma_n] V_0 \to \Z_p[\Gamma_n] E_0^+$. Note that $S(X_n,\Z_p) = \Z_p[\Gamma_n] S(X_0,\Z_p)$, but $H_1(X_n,\Z_p) \subset H_1(X_0,\Z_p)$ will be a submodule with quotient $H_1(X_0,\Z_p)/H_1(X_n,\Z_p) \cong \Gal(X_n/X_0) = \Gamma_n$.

\subsection{The Module}
Now we patch together the finitary data from the previous section into a $\Lambda$-module. Write $C(\infty) = \Lambda E_0^+$ and $D(\infty) = \Lambda V_0$, noting that we still have the maps $d$ and $s$ on this level. Write $H(\infty) = \ker\left(d \colon C(\infty) \to D(\infty) \right)$ and $S(\infty) = \im\left(s \colon D(\infty) \to C(\infty) \right)$. Since these are $\Lambda$-submodules of $C(\infty)$, we can define $J(\infty) = C(\infty)/\left(S(\infty) + H(\infty) \right)$.

\begin{rem}
For the sake of intuition, let $X_\infty = \varprojlim X_n$ be the profinite graph at the top of the tower. We can look at $C(\infty) = \Lambda E_0^+$ as some completion of $\Z_p E_\infty^+$, with $S(\infty)$ some completed group of stars and $H(\infty)$ some completed version of $H_1(X_\infty,\Z_p)$. In this lens, we see that $J(\infty)$ should be (a completed version of the $p$-part of) the Jacobian of $X_\infty$; however, it is not quite so simple on the level of the finite graphs. In fact, while this could be thought of as $\J_e(X_\infty)[p^\infty]$, the corresponding $\Z_p[\Gamma_n]$-modules $J(\infty) \otimes_\Lambda \Z_p[\Gamma_n]$ at each finite step of the tower will be an extension of $\J(X_n)[p^\infty]$ by $\Gamma^{(n)}$. The discrepancy comes from the fact that $H_1(X_n,\Z_p)/H(\infty) \cong \Gal(X_\infty/X_n) = \Gamma^{(n)} \cong \Z_p$.
\end{rem}

Write $J(n) = J(\infty) \otimes_\Lambda \Z_p[\Gamma_n]$, and similarly define $C(n)$, $D(n)$, $H(n)$, and $S(n)$. Since tensor products preserve cokernels, we have that $J(n) \cong C(n)/[H(n)+S(n)]$. Note that $C(n) \cong \Z_p E_n^+$, $D(n) = \Z_p V_n$, and $S(n) = S(X_n,\Z_p)$ by the above argument, but $H(n) \subset H_1(X_n,\Z_p)$ is just a submodule. We get an exact sequence
\begin{equation}
	0 \to H(n) \to H_1(X_n,\Z_p) \to J(n) \to \J(X_n) \otimes \Z_p \to 0.
\end{equation}
Since $\J(X_n)$ is finite, we have that $\J(X_n) \otimes \Z_p \cong \J(X_n)[p^\infty]$ is the $p$-power torsion subgroup of $\J(X_n)$. We can also replace $H_1(X_n,\Z_p)$ by the quotient $H_1(X_n,\Z_p)/H(n) \cong \Gamma^{(n)}$ to obtain the following short exact sequence:
\begin{equation}
	0 \to \Gamma^{(n)} \to J(n) \to \J(X_n)[p^\infty] \to 0.
\end{equation}
We have proven the following.

\begin{thm}\label{thm:themodule}
There is a module $J(\infty)$ over the Iwasawa algebra $\Lambda$ whose coinvariants $J(n) = J(\infty) \otimes_\Lambda \Z_p[\Gamma_n]$ are isomorphic to an extension of the $p$-power torsion subgroup of the Jacobian $\J(X_n)[p^\infty]$ by the group $\Gamma^{(n)}$.
\end{thm}

\begin{rem}\label{rem:actiononliftedspanningtrees}
For any cover $\pi \colon \tilde{X} \to X$ with abelian Galois group $G$ we can build a natural $G$-action on $\J(\tilde{X})$; by the same arguments above, taking coinvariants will give us an extension of $\J(X)$ by $G$. Call this $\J(\pi)$.

In Section \ref{sec:spanningtrees}, we referenced the fact that the set of spanning trees on a graph $X$ is a torsor for the Jacobian $\J({X})$, which means that we can count spanning trees on ${X}$ by finding $\#\J({X})$. One may ask what $\J(\pi)$ counts.

It is not too difficult to see the fact that $\#\J(\pi)$ is the number of lifts of spanning trees on $X$ to (non-spanning) subtrees of $\tilde{X}$: $\#\J(B)$ is the number of spanning trees on $B$, and $\# G$ is how many different choices we have for a lift of any chosen basepoint, so there are $(\#\J(X))(\#G) = \#\J(\pi)$ many lifts. However, it seems more difficult to define such an action intrinsically on $\tilde{X}$, due to the fact that rotor-routing essentially needs the trees to be spanning in order for the process to be well-defined.

We ask a similar question here to the question in Remark \ref{rem:SpanningTreesAreATorsor}: Is there a natural action of $\J(\pi)$ on the set of lifts of spanning trees from $X$ to $\tilde{X}$?
\end{rem}

\section{Asymptotics}
\label{sec:examples}

\subsection{Difficulties with the Asymptotics}\label{sec:asymp}
The asymptotics for $\#\J(X_n)$ are complicated by the fact that $\#J(n) = \infty$ for all $n$. We need to leverage the connection between different homology groups in the tower to extract the finite groups $\J(X_n)[p^\infty]$ from the infinite groups $J(n)$. It is difficult to do in generality, which is why we dedicate the next section to an example. In this section, we give an idea of how one might deal with this discrepancy in a uniform way.

Let $(X_0,\Gamma,\alpha)$ be a voltage assignment as in the beginning of Section \ref{sec:CSHinTowers}, with the $\Z_p$-tower $\{X_n\}_{n=0}^\infty$ constructed as described there. Pick a base vertex $v_0$ on $X_0$, and let $v_1$ be a chosen lift of $v_0$ to $X_1$. Since $X_1$ is connected, there is a path $\delta_1=e_1\cdots e_\ell$ that connects $v_1$ to $\gamma v_1$. This projects to a path $\delta_0$ on $X_0$, and $\delta_0 \in H_1(X_0,\Z_p)$.

This $\delta_0$ will generate $H_1(X_0,\Z_p)/H(0)$ as a $\Z_p$-module. In fact, if we let $\delta_n$ be any lift of $\delta_0$ to a path on $X_n$, then $\sum_{i=0}^{p^n-1} \gamma^i \delta_n$ will be a loop on $X_n$, and it will generate $H_1(X_n,\Z_p)/H(n)$ as a $\Z_p$-module. Writing $N_n(\gamma) = \sum_{i=0}^{p^n-1}\gamma^i$, we have argued that $\J(X_n)[p^\infty] \cong J(n)/\Z_p N_n(\gamma)\delta_0$. In the next section, we describe how to use this isomorphism to get asymptotics for $\#\J(X_n)[p^\infty]$ in an example.

\subsection{The Example} \label{sec:theexample}
Let $X_0 = B_{k}$ be the bouquet of $k$ loops, with $V_{0} = \{*\}$ and $\vec{E}_0 = \{e_0,\overline{e}_0,\dots,e_{k-1},\overline{e}_{k-1}\}$. Choose the orientation $E_0^+ = \{e_0,\dots,e_{k-1}\}$. Then fix a prime $p$. Let $\Gamma$ be a group isomorphic to $\Z_p$, let $\gamma$ be a chosen topological generator, and consider the $\Z_p$ tower of graphs determined by the voltage assignment $\alpha \colon E_0^+ \to \Gamma$ with $\alpha(e_i) = \gamma$ for each $i$. Consider the projection $f_n \colon \Gamma \to \Gamma_n = \Gamma/\langle\gamma^{p^n}\rangle$. Replacing $\alpha$ by $\alpha_n = f_n \circ \alpha$ gives rise to a voltage graph $(X_0,\Gamma_n,\alpha_n)$ whose derived graph we will call $X_n$.

In fact, $X_n$ is fairly easy to describe, as is the projection map $\pi_n \colon X_n \to X_0$. $X_n$ has $p^n$ vertices, indexed by the elements of $\Z/p^n\Z$, and each vertex $v_i$ is connected to the vertex $v_{i+1}$ by $k$ edges $e_{i,j}$ for $j=0,\dots,k-1$. The projection map $\pi_n$ sends every vertex of $X_n$ to the unique vertex of $X_0$, and sends the edge $e_{i,j}$ of $X_n$ to the edge $e_j$ of $X_0$.

We have $C(\infty) = \Lambda\{e_0,\dots,e_{k-1}\}$ is the free $\Lambda$ module on the basis $E_0^+$. $S(\infty)$ is spanned over $\Lambda$ by the star into $v_1$, which is $e_0 + \dots + e_{k-1} - \gamma e_0 - \dots - \gamma e_{k-1}$. Finally, $H(\infty)$ is spanned over $\Lambda$ by the elements $e_i - e_0$, where $i=1,\dots,k-1$. There is a map $C(\infty) \to \Lambda/(k-k\gamma)$ which sends each edge $e_i$ to $1$; the kernel certainly contains $H(\infty)$ since it kills the differences $e_i - e_0$. Once those differences are killed, the generator of $S(\infty)$ becomes $k-k\gamma$, so the map factors through $J(\infty)$; it is not difficult to show that the induced map is an isomorphism $J(\infty) \xrightarrow{\sim} \Lambda/(k-k\gamma)$.

As we noted before, $H(n) = H(\infty) \otimes_\Lambda \Z_p[\Gamma_n]$ is not all of $H_1(X_n,\Z_p)$, and in fact the quotient $H_1(X_n,\Z_p)/H(n)$ will be isomorphic to $\Z_p$. Following the ideas in Section \ref{sec:asymp}, we notice that it is spanned over $\Z_p$ by the element $N_n(\gamma)e_0$, where $N_n(\gamma) = 1+\gamma+\dots+\gamma^{p^n-1}$. Writing $J(n) = J(\infty) \otimes_\Lambda \Z_p[\Gamma_n]$, we have 
\begin{equation}
\J(X_n)[p^\infty] \cong J(n)/H_1(X_n,\Z_p) \cong \Z_p[\Gamma_n]/(k-k\gamma,1+\gamma+\dots+\gamma^{p^n-1}).
\end{equation}

If $k$ is prime to $p$, the calculation is easier. $k$ is a unit in $\Z_p$, so
\begin{equation}
\begin{array}{rcl}
	\J(X_n)[p^\infty] &\cong& \Z_p[\Gamma_n]/(k-k\gamma,1+\dots+\gamma^{p^n-1})\\[6pt]
	&\cong& \Z_p[\Gamma_n]/(1-\gamma,1+\dots+\gamma^{p^n-1}) \\[6pt]
	&\cong& \Z_p/(p^n).
\end{array}
\end{equation}
Thus, when $p \nmid k$, we have $\J(X_n)[p^\infty] \cong \Z/p^n\Z$, so $\#\J(X_n)[p^\infty] = p^n$.

If $p \mid k$, the relationship is more complicated. Let $u = v_p(k)$, so that $p^u$ is the largest power of $p$ that divides $k$. Then we instead write
\begin{equation}
\begin{array}{rcl}
	\J(X_n)[p^\infty] &\cong& \Z_p[\Gamma_n]/(k-k\gamma,1+\dots+\gamma^{p^n-1})\\[6pt]
	&\cong& \Z_p[\Gamma_n]/(p^u(1-\gamma),1+\dots+\gamma^{p^n-1}).
\end{array}
\end{equation}
We have a map $f = (f_1,f_2) \colon \Z_p[\Gamma_n] \to \Z_p \times (\Z/p^u\Z)[\Gamma_n]$,
\begin{equation}
	\sum_{i=0}^{p^n-1} a_i \gamma^i \mapsto \left(\sum_{i=0}^{p^n-1} a_i, \sum_{i=0}^{p^n-1} [a_i] \gamma^i \right).
\end{equation}
Here the $[a_i]$ in brackets is the class of $a_i$ in $\Z/p^u\Z$. The kernel of the first component $f_1$ is the ideal generated by $\gamma - 1$, and the kernel of $f_2$ is the ideal generated by $p^u$, so the kernel of $f = (f_1,f_2)$ is $(\gamma - 1) \cap (p^u) = (p^u(\gamma-1))$. Then the image is isomorphic to $J(n)$. This is the set of tuples
\begin{equation}
\operatorname{im}(f) = \left\{\left(m,\sum [a_i]\gamma^i \right) \mid m \equiv \sum [a_i] \pmod{p^u} \right\}
\end{equation}
The group $\J(X_n)[p^\infty]$ is a quotient of $\operatorname{im}(f)$ by the $\Z_p[\Gamma_n]$-submodule generated by $f(N_n(\gamma)) = (p^n,N_n(\gamma))$. Each coset of this submodule has a unique representative $\left(m,\sum \alpha_i \gamma^i\right)$ with $m$ an integer, $0 \leq m < p^n$. Thus, at least as a set, we have that $\J(X_n)[p^\infty]$ is isomorphic to
\begin{equation}
\left\{ \left(m,\sum_i \alpha_i \gamma^i\right) \in  \Z/p^n\Z \times (\Z/p^u\Z)[\Gamma_n] \mid m \equiv \sum_i \alpha_i \pmod{p^u} \right\}.
\end{equation}
Let $n \geq u$. For each of the $(p^u)^{p^n}$ elements $\sum_{i=0}^{p^n-1} \alpha_i \gamma^i$ of $(\Z/p^u\Z)[\Gamma_n]$, there will be $p^{n-u}$ choices for $m$. Thus $\#\J(X_n)[p^\infty] = p^{n-u} (p^u)^{p^n} = p^n(p^u)^{p^{n}-1}$.

Note that in the case $u=0$, i.e., when $p \nmid k$, we recover the simpler result that $\#\J(X_n)[p^\infty] = p^n$.

\begin{rem}
In fact, in this example, we can actually count the spanning trees of $X_n$. Recall that $X_n$ is a graph with $p^n$ vertices labeled by the elements of $\Z/p^n\Z$, where each vertex $v_i$ is connected to the vertex $v_{i+1}$ by $k$ distinct edges. We can specify a spanning tree with the following choices: first, a vertex $v_{i_0}$, and then, for all $i \neq i_0$, a choice of one of the $k$ edges that connects $v_i$ to $v_{i+1}$. This gives us a total of $p^n k^{p^n-1}$ choices. The $p$-part of this number is obtained by replacing $k$ with $p^u$, which gives us $p^n (p^u)^{p^n-1}$ as the target number. This is exactly what we calculated for $\#\J(X_n)[p^\infty]$.
\end{rem}

\begin{rem}
Though we have to make the adjustment with $H_1(X_0,\Z_p)$, we do still end up getting an asymptotic of the standard type for $\#\J(X_n)[p^\infty]$. For $n \geq u$, we have $\#\J(X_n)[p^\infty] = p^{up^n + n - u}$, which is of the form $p^{\lambda p^n + \mu n + \nu}$ for $\lambda = u$, $\mu = 1$, and $\nu = -u$.
\end{rem}

\subsection{The $L$-function}

As in \cite[Section 3.1]{HMSV}, we write $L(X_0,\chi,u)\inv = L^*(X_0,\chi,1)(u-1)^{r(\chi)} + \dots$. We can compute the values $L^*(X_0,\chi,1)$ for varying nontrivial characters $\chi$ of $\Gamma$.

By \cite[Theorem 18.15]{Terras}, we have that $L^*(X_0,\chi,1) = (-2)^{r-1}\det(D - A_\chi)$, where $r$ is the rank of the fundamental group of $X_0$, and $D$ is the degree matrix. Unfolding the definition of $A_\chi$, we obtain a matrix $A_\Gamma$ with entries in $\Z[\Gamma]$ such that $L^*(X_0,\chi,1) = \det(\chi(D-A_\Gamma)) = \chi(\det(D-A_\Gamma))$. Thus, the $p$-adic $L$-function should be $\det(D-A_\Gamma)$.

Recall that, for each $v_0 \in V_0$, we can choose a compatible system of lifts $v_n \in V_n$. Then for each $n$, let $A_{\Gamma_n}$ be a matrix whose rows and columns are indexed by the vertices of $X_0$, and whose $v_0,w_0$ entry is
\begin{equation}
	a_{v_0,w_0} = \sum_{\sigma \in \Gamma_n} \#\left\{ \text{edges connecting }v_n\text{ to }\sigma(w_n) \right\}\sigma.
\end{equation}
These $A_{\Gamma_n}$ form a compatible system of elements of $M_{\#V_0}(\Z_p[\Gamma_n])$, and so they determine an element $A_\Gamma$ of $\varprojlim M_{\#V_0}(\Z_p[\Gamma_n]) = M_{\#V_0}(\Lambda)$. This is the $A_\Gamma$ in the formula above for the $p$-adic $L$-function.

When $X_0 = B_k$, we have $\#V_0 = 1$, so the matrices are all $1 \times 1$. Using the voltage assignment as in Section \ref{sec:theexample}, we find that $A_\Gamma = k\gamma + k\gamma\inv$. As $D = 2k$, we have $L^*(X_0,-,1) = (-2)^{r-1}(2k-k\gamma-k\gamma\inv)$.

If $p=2$, the factor of $(-2)^{r-1}$ complicates things, so let $p \neq 2$. Up to unit, we can write $L^*(X_0,-,1) \approx (2k-k\gamma-k\gamma\inv) \approx k(1-\gamma)^2$. Note that this is not the generator of the characteristic ideal we found in Section \ref{sec:theexample}; however, it differs by a factor of the augmentation ideal $(1-\gamma)$, which is exactly what is predicted in \cite[Theorem 5.2, Remark 5.3]{KM}.

\appendix


\bibliographystyle{amsalpha}
\bibliography{AycockIwasawaViaEdgesBib}

\end{document}